\documentclass[12pt]{amsart}

\usepackage{amssymb,amsthm,amsmath,float, xcolor}

\usepackage[all,cmtip]{xy}

\newtheorem{theorem}{Theorem}
\newtheorem{definition}{Definition}
\newtheorem{lemma}{Lemma}
\newtheorem{remark}{Remark}
\begin{document}

\title[A characterization of the $n$-torus]{A characterization of the $n$-dimensional torus via intrinsically harmonic forms}

\author{Elizeu França}
\address{Universidade Federal do Amazonas \\
69103-128 Itacoatiara, AM, Brazil}
\email{elizeufranca@ufam.edu.br}

\author{Douglas Finamore}
\address{Departament of Mathematics - ICMC, University of S\~ao Paulo}
\email{douglas.finamore@usp.br}


\begin{abstract}
    The $n$-torus is the unique closed manifold supporting a set of $n$ linearly independent closed $1$-forms.
    In this paper we improve on this result and show that the torus is the unique closed $n$-dimensional manifold supporting a linearly independent set consisting of $(n-1)$ closed $1$-forms whose product determines a nonzero cohomological class. 
\end{abstract}

\maketitle

\section*{Acknowledgements}
The second author was funded by Brazilian Coordena\c c\~ao de Aperfei\c coamento de Pessoal de N\'ivel Superior (CAPES), grant PROEX-11377206/D. 
Both authors are deeply grateful to Gabriel Ponce, for introducing them to one another and making this joint work possible.

\section{Introduction}\label{sec:Intro}
\par Suppose $M$ is an $n$-dimensional closed manifold supporting a collection of $n$ linearly independent closed $1$-forms.
Then $M$ is diffeomorphic to the $n$-dimensional torus $\mathbb{T}^n$.
One way to see this is the following: using Tischler's argument \cite{tischler_fibering_1970}, the collection of linearly independent forms provide a surjective submersion $M \to \mathbb{T}^n$.
Since $M$ is compact, such a surjection must be proper, and therefore it is a covering map \cite[Lemma 3]{ho_note_1975}.
As is well-known from Lie group theory, a compact covering of $\mathbb{T}^n$ is diffeomorphic to $\mathbb{R}^n/H$, where $H \subset \pi_1(\mathbb{T}^n)=\mathbb{Z}^n$ is a subgroup of the form 
\begin{equation*}
    H \approx m_1\mathbb{Z}\times\cdots \times m_n\mathbb{Z},
\end{equation*}
with $m_i\in \mathbb{Z}$.
Therefore, $M$ and $\mathbb{T}^n$ are diffeomorphic.
\par In this note, we prove the following result.
\begin{theorem}\label{main_thm} 
    Let $M$ be a closed $n$-dimensional manifold supporting $(n-1)$-linearly independent closed $1$-forms $\lambda_1,\cdots\!,\lambda_{n-1}$. 
    If 
    \begin{equation*}
        \omega := \lambda_1\wedge\cdots\wedge\lambda_{n-1}
    \end{equation*}
    determines a nonzero cohomological class, then it is intrinsically harmonic. In particular, there exists a closed $1$-form $\eta$ such that the set $\{\lambda_1, \cdots\!,\lambda_{n-1},\eta\}$ is linearly independent, and therefore $M$ is diffeomorphic to $\mathbb{T}^n$.
\end{theorem}
\par Recall that harmonicity of differential forms is defined in terms of a Riemannian metric tensor on the ambient manifold.
A natural question is then: given a closed form $\omega$ on a compact manifold $M$, is there a Riemannian metric $g$ on $M$ such that $\omega$ is harmonic with respect to $g$?
If such a metric exists, $\omega$ is said to be \textit{intrinsically harmonic}.
\par The problem of obtaining an intrinsic characterization of harmonic forms was first placed by Calabi \cite{calabi_intrinsic_1969}. 
He showed that, under suitable conditions on its zero-set, an $1$-form $\omega$ is intrinsically harmonic if and only if it is \emph{transitive}. 
Transitivity means that every point that is not a zero for $\omega$ is contained in an embedded circle, on which the restriction of $\omega$ never vanishes.
\par Volkov generalized Calabi's characterization of intrinsically harmonic $1$-forms by showing that a closed $1$-form is intrinsically harmonic if, and only if, it is transitive and there is a neighborhood of its zero set where it is harmonic  \cite{volkov_characterization_2008}. 
\par In \cite{honda_harmonic_1996}, Honda proved a ``dual'' version of Calabi-Volkov's result, showing that a closed $(n-1)$-form is intrinsically harmonic, under suitable conditions on its zero-set, provided that it is transitive.
Similarly to the $1$-dimensional case, a $k$-form $\omega$ is transitive if any ``regular point'' is contained in a closed $k$-dimensional submanifold on which  $\omega$ restricts to a volume form.
\par Using foliation theory, one can show that a nowhere-vanishing closed $1$-form on a closed manifold is transitive, hence intrinsically harmonic.
The same conclusion is not true, however, for forms of degree $(n-1)$.
As an example, consider de Hopf's fibration $H:\mathbb{S}^3 \to \mathbb{S}^2$, and let $\Omega$ a volume form on $\mathbb{S}^2$.
Then $H^*\Omega$ is a closed nowhere vanishing $2$-form on $\mathbb{S}^3$ which can not be harmonic for any Riemannian metric $g$, because nonzero harmonic forms determine nonzero cohomological classes, however $H^2(\mathbb{S}^3)=0$. 
The main motivation behind this paper, as well some previous relate work, is to show that nonvanishing closed $(n-1)$-forms are intrinsically harmonic, provided they determine a nonzero cohomological class.
% Generally, in a orientable circle bundle $\xi=(\pi,M, B)$, given a volume form $\Omega$ on $B$, the form $\pi^*\Omega$ is intrinsically harmonic if, and only if, $\xi$ is a flat \cite[Theorem 1]{franca_intrinsically_nodate}.
In this paper, we show that the form $\omega=\lambda_1\wedge\cdots\wedge\lambda_{n-1}$ from Theorem \ref{main_thm} represents a nonzero cohomological class if, and only if, it is intrinsically harmonic.
\par It is known that, given a volume form $\Omega$ on an orientable manifold, every closed $(n-1)$-form corresponds to a unique vector field whose flow preserves $\Omega$ (cf. Section \ref{rank1_form}).
Conversely, any volume-preserving flow induces a closed $(n-1)$-form via interior multiplication.
The fixed points of the flow are exactly the singular points of the associate form.
In Section \ref{sec:k-forms} we present the following criteria for harmonicity of such forms.
\begin{theorem}\label{pre_char} 
    Let $M$ be a closed orientable manifold.
    A nowhere-vanishing volume-preserving flow defined on $M$ admits a global cross-section if and only if the induced closed nowhere-vanishing $(n-1)$-form is intrinsically harmonic.
\end{theorem}
\noindent By a \emph{global cross section} we mean a closed $(n-1)$-dimensional submanifold everywhere transverse to the flow and cutting every orbit.
Necessary and sufficient conditions for the existence of a global cross-section are given in \cite{fried_geometry_1982} and \cite{schwartzman_asymptotic_1957}, for instance.
In \cite{schwartzman_asymptotic_1957}, Schwartzman relates global cross sections to the notion of \emph{asymptotic cycles}, which are real homology classes defined for each flow-invariant measure.
A global cross-section is determined by an integral $1$-dimensional cohomology class that is positive on every asymptotic cycle.
In the case of recurrent flows (cf. Definition \ref{recur_flow}), Schwartzman shows the existence of a single asymptotic cycle.  
\par In our case of interest, $\Omega$ is a volume form on a manifold $M$, associated to the measure $\mu(B) := \int_B\Omega$.
Given a smooth vector field  $X$ preserving $\Omega$ (and therefore also $\mu$), the asymptotic cycle $A_\mu$ associated with the measure $\mu$ is completely determined by the differential form $\omega:= \iota_X\Omega$. 
In particular, if $\omega$ determines a nonzero cohomological class, then $A_\mu\neq 0$ in $H_1^{dR}(M)$ (cf. Lemma \ref{canonicalCycleVolPreservingFlow}).
If, in addition, $X$ is (uniformly) recurrent, the existence of a global cross-section depends only on the homological class $[\omega]=A_\mu$. 
\par Due to Poincar{{\'e}}'s Recurrence Theorem, a volume-preserving flow on a compact manifold is pointwise recurrent.
It is only natural to ask under what conditions one can obtain uniform recurrence for such flows.
The recurrence property is equivalent to the existence of a suitable function sequence converging uniformly to the identity.
In general, a uniformly converging sequence of functions from a compact metric space to itself determines an (uniformly) equicontinuous set.
Thus, a natural follow-up question is when, given a flow $\Phi$, the collection $\{\Phi^t\}_{t\in\mathbb{R}}$ is an (uniformly) equicontinuous set of functions. 
Such flows are called \emph{equicontinuous} (cf. Definition \ref{equi_fam}), and these satisfy the recurrence property.
For instance, flows induced by Killing fields on Riemannian manifolds are equicontinuous.
Therefore, in order to obtain Theorem \ref{main_thm}, we show that the flow induced by $\omega$ is Killing for some Riemannian metric on $M$.
With this we can guarantee the recurrence of the flow and, using Schwartzman Theory, the existence of a global cross-section.
Then Theorem \ref{main_thm} follows as a consequence of Theorem \ref{pre_char}.\smallskip

\noindent \textbf{Some context}.
One of the primary goals of the first author's recent research has been to prove the ``dual" of Tischler's theorem, that is, to show that an orientable manifold supporting a closed nowhere-vanishing $(n-1)-$form 
 fibers over $\mathbb{S}^1$. In order to that, some sort of classification theorem for volume-preserving flows admitting global cross-sections is required.
Ideally, we would like to show that if a $(n-1)$-form, which is naturally associated with such a flow, represents a nonzero cohomological class, then the desired cross-section exists. 
\par In this paper, we look at a particular case where the flow naturally associated with a intrinsically harmonic $(n-1)$-form is Killing with respect to some metric on the ambient manifold.
With such conditions, we are able to show that the flow is recurrent, and what is more, give a partial classification result in the form of Theorem \ref{int_har_char2}.
The flow's recurrence can then be used in conjunction with Schwartzman's theory of asymptotic cycles to obtain a global cross-section, and thus a ``complementary form."
As a consequence of these results, we obtain the characterization of torus stated in Theorem \ref{main_thm}.
\par Of course, other possible restrictions on the flow exist as well. 
For instance, in a recent study \cite{francca2023pointwise}, we explored the scenario where the flow is pointwise periodic and investigated whether this implies equicontinuity or recurrence.
The results showed that even for only continuous flows, equicontinuity implies periodicity, indicating it is it in fact a robust condition. \smallskip

\noindent \textbf{Organization}.
In Section \ref{sec:Prelim} we briefly discuss required notions and results on induced flows, asymptotic cycles and equicontinuous actions. Section \ref{sec:k-forms} relates intrinsically harmonic forms to forms of constant rank, and culminates with the proof of Theorem \ref{pre_char}. In Section \ref{sec:vol_pre_char} we provide a characterization for volume-preserving flows supporting cross-sections, which is then used to prove Theorem \ref{main_thm} in Section \ref{sec:tor_char}.

\begin{center}
    \emph{ Throughout this work, all manifolds, differential forms, flows and vector fields are considered to be $C^\infty$}.
\end{center}

\section{Preliminaries}\label{sec:Prelim}
 \subsection{The induced flow of an $(n-1)$-form on an orientable manifold}\label{rank1_form}
Let $\omega$ be a $k$-form on a $n$-dimensional manifold $M$.
We will denote by $\ker\omega$ the distribution on $M$ given by
\begin{equation*}
	x\mapsto \ker\omega\rvert_x=\{v\in T_xM; \iota_v\omega=0\}.
\end{equation*}
\begin{definition}[rank of a differential form]
    The \emph{rank} of $\omega$ at a point $x\in M$ is defined by $r_x=n-\dim\ker_x$.
    A form $\omega$ is said to be of \emph{constant rank} if $r_x$ is independent of $x$. 
\end{definition}
Every closed form of constant rank $k$ defines a $k$-dimensional foliation $\mathcal{F} $ on its ambient manifold.
Indeed, let $\omega$ be a closed form of rank $k$.
Given $X,Y\in\ker\omega$, we use the identity $\iota_{[X,Y]}=[\mathcal{L}_X,\iota_Y]$ and obtain
\begin{equation*}
	\iota_{[X,Y]}\omega=\mathcal{L}_X \iota_Y\omega - \iota_Y\mathcal{L}_X\omega=\mathcal{L}_X \underbrace{\iota_Y\omega}_{0} - \iota_Y(\iota_X \underbrace{d\omega}_0 + d\underbrace{\iota_X\omega}_0) = 0,
\end{equation*}
that is, the distribution $\ker\omega$ is involutive.
Hence, using Frobenius' Theorem, we conclude that $\ker\omega$ is tangent to a foliation $\mathcal{F} $ of codimension $k$. 
\par For the particular case of an orientable manifold, we can show that every closed $(n-1)$-form has rank $(n-1)$, and that the associated foliation is the flow of a volume-preserving field. 
What is more, to every volume preserving field we can associate a closed $(n-1)$-form, yielding the following.
\begin{theorem}
    On an orientable manifold, every closed $(n-1)$ form has constant rank $(n-1)$. Moreover, there is a bijection between the sets of volume preserving fields and closed $(n-1)$-forms.
\end{theorem}
To see this, let $M$ be an $n$-dimensional orientable manifold equipped with a volume form $\Omega$.
\par First, we note that given any volume preserving field $X$, the $1$-form $\omega = \iota_X\Omega$ is closed, since by Cartan's Formula
\begin{equation*}
    \mathrm{d}\omega = \mathrm{d}(\iota_X\Omega) = \mathcal{L}_X\Omega - \iota_X\mathrm{d}\Omega = 0.
\end{equation*}
Note that $\iota_X\omega = \iota_X\iota_X\Omega =0$, so that $X$ is tangent to $\ker\omega$.
On the other hand, there can be no vector $Y \in \ker\omega$ linearly independent to $X$, since this would imply $\iota_Y\iota_X\Omega = 0$, contradicting the fact that $\Omega$ is a volume form.
So indeed, $\ker\omega$ is one-dimensional at every point, and $\omega$ has constant rank $(n-1)$.
\par It remains to show that every closed $(n-1)$-form on $M$ has rank $(n-1)$ and its associated to a volume preserving field $X$.
In order to do that, we will focus in the particular case were the $(n-1)$-form $\omega$ consists of the product of $n-1$ closed linearly independent $1$-forms, 
\begin{equation*}
    \omega = \lambda_1 \wedge \cdots \wedge \lambda_{n-1},
\end{equation*}
since this is the relevant scenario in Section \ref{sec:tor_char}.
We remark, however, that our arguments can be easily adapted to the case of general closed $(n-1)$-forms using coordinate patches and partitions of unity.
\par At each point $x \in M$, the linear functionals $\lambda_i\rvert_x: \mathrm{T}_xM \to \mathbb{R}$ are linearly independent, hence the intersection of their kernels consists of a line $\cap\ker\lambda_i\rvert_x \subset \mathrm{T}_xM$.
Moreover, one can find a unique set of linearly independent vectors $\{X_1\rvert_x, \cdots\!, X_{n-1}\rvert_x\} \subset \mathrm{T}_xM$ satisfying relations
\begin{equation}\label{dualfields}
\lambda_i\rvert_x(X_j\rvert_x) = \delta_{ij}.
\end{equation}
We can define linearly independent vector fields on $M$ by setting
\begin{align*}
X_i: M & \to \mathrm{T} M\\
x &\mapsto X\rvert_x.
\end{align*}
These fields are as smooth as the $1$-forms $\lambda_i$.
\begin{theorem}
The equality 
\begin{equation*}
\ker\omega\rvert_x = \bigcap_i\ker\left(\lambda_i\rvert_x\right)
\end{equation*}
holds for every $x \in M$.
In particular, $\ker\omega = \cap\ker\lambda_i$ is a line distribution on $TM$, and therefore integrable.
\end{theorem}
\begin{proof}
	Given $X \in \mathrm{T}_xM$, suppose $X \in \cap\ker\lambda_i\rvert_x$. 
	Then
	\begin{equation*}
   		\iota_X\omega\rvert_x = \sum_{i=1}^{n-1}(-1)^{i-1}\underbrace{\lambda_i\rvert_x(X)}_{ 0 } \lambda_1\rvert_x \wedge \cdots \wedge \widehat{\lambda_i\rvert_x} \wedge \cdots \wedge \lambda_{n-1}\rvert_x = 0,
	\end{equation*}
	hence $\cap\ker\lambda_i\rvert_x \subset \ker\omega\rvert_x$.
	We claim that the set $\left(\ker\omega\rvert_x\right) \setminus \left(\cap\ker\lambda_i\rvert_x\right)$ is empty.
	Indeed, if $X$ is a vector in $\ker\omega\rvert_x$ but not in $\cap\ker\lambda_i\rvert_x$, then $\lambda_j\rvert_x(X) \neq 0$ for some $1 \leq j \leq n-1$.
	Hence, evaluating $\iota_X\omega\rvert_x$ at the linearly independent set $\{X_1\rvert_x, \cdots\!, \widehat{X_j\rvert_x}, \cdots\!, X_{n-1}\rvert_x\}$ of vectors satisfying Equations \ref{dualfields}, we obtain
	\begin{equation*}
    	0 = \iota_X\omega(X_1\rvert_x, \cdots\!, \widehat{X_j\rvert_x}, \cdots\!, X_{n-1}\rvert_x) = (-1)^{j-1}\lambda_j\rvert_x(X),
	\end{equation*}
	a contradiction.
	Thus $\ker\omega\rvert_x \subset \cap\ker\lambda_i\rvert_x$, and equality follows.
\end{proof}
\par Since the distribution $\ker\omega$ is one-dimensional at every point, it is spanned by a single nonvanishing vector field, say $X_0$.
The underlying foliation $\mathcal{F} $ is exactly the unparameterized flow of $X_0$.
\par Moreover, the set $\{X_0, X_1, \cdots\!, X_{n-1}\}$ is a global frame to $M$, which we can use define a $1$-form $\lambda_0$ on $M$ by imposing 
\begin{equation*}
	\lambda_0(X_i) = \delta_{0i}
\end{equation*}
and extending it linearly.
It follows that $\Omega := \lambda_0\wedge\cdots\wedge\lambda_{n-1}$ is a volume form on $M$, and that $\omega = \iota_{X_0}\Omega$.
As a consequence, $\mathcal{L}_{X_0}\Omega = 0$, so $X_0$ is volume-preserving, as we wanted.
\begin{remark}\label{lambda0notclosed}
	If $\lambda_0$ were closed, then $\{\lambda_0, \cdots\!, \lambda_{n-1}\}$ would be a linearly independent set of closed $1$-forms, and Theorem \ref{main_thm} would be proved.
 However, in general $\mathrm{d}\lambda_0 \neq 0$.
 To see this, note that, for $i = 1, \cdots\!, n-1,$ each vector field $X_i$ is \emph{foliate}, that is, it satisfies 
\begin{equation*}
[X_0, X_i] \in \ker\omega.
\end{equation*}
Indeed, using the invariant formula for the derivative of $1$-forms, we obtain
\begin{equation*}
     \lambda_j([X_0, X_i]) = X_0 \underbrace{ \lambda_j(X_i)}_{\delta_{ij}} - X_i\underbrace{\lambda_j(X_0)}_{ 0 } - \underbrace{\mathrm{d}\lambda_j}_{ 0 }(X_0, X_i) = 0.
\end{equation*}
Thus,
\begin{equation*}
[X_0, X_i] \in \bigcap_{j=1}^{n-1}\ker\lambda_j = \ker\omega.
\end{equation*} 
Hence, given a field $N$ tangent to the sub-bundle spanned by the frame $\{X_1, \cdots\!, X_{n-1}\}$, we have $[N, X_0] = fX_0$ and therefore
\begin{equation*}
\mathrm{d}\lambda_0(N,X_0) = N\lambda_0(X_0) - X_0\lambda_0(N) - \lambda_0([N, X_0]) = f.
\end{equation*}
\end{remark}

\subsection{Equicontinuity and uniform recurrence}
\par Poincar{{\'e}}'s Recurrence Theorem states that every point of an orientable compact manifold is recurrent under a volume-preserving flow.
We want to show that if one asks for a stronger condition on the flow, namely equicontinuity, then one gets a stronger form of recurrence as well.
\begin{definition}\label{equi_fam}
A set $T $ of continuous functions from a topological space $X$ into a metric space $(Y,d)$ is  called \emph{equicontinuous} provided corresponding to each $x_0\in X$ and each
$\epsilon>0$, there is a neighborhood $U$ of $x_0$ in $ X $ such that $d(f(x_0),f(x)) < \epsilon$ for all $f \in T$ and all $x \in U$.
\end{definition}
\par In particular, a flow on a metric space $(X,d)$ is equicontinuous if the family $T = \{t: X \to X \}_{t \in \mathbb{R}}$ is equicontinuous.
Explicitly, given $\epsilon>0$ there exists $\delta>0$ such that 
\begin{equation*}
    d(x,y)<\delta \implies d(xt, yt)<\epsilon 
\end{equation*}
for all $t\in\mathbb{R}$.
\begin{definition}\label{recur_flow}
A flow  on a metric space $(X,d)$ is said to be recurrent if there exists a sequence $t_n \to \infty$ such that $\lim_{n}\sup_{x}d(x,xt_n)=0$.
\end{definition}
\par Let $X$ be a  topological space, and a $(Y,d)$ be a metric space.
The space $C(X,Y)$ of all continuous functions of $X$ to $Y$ has a natural topology induced by the metric $d$.
It is called \emph{the topology of uniform convergence}
\footnote{
    More generally, when $Y$ is a uniform space, one can define this notion analogously.
} 
and has as sub-basis for open subsets the sets
\begin{equation*}
    B_\epsilon(f)=\{g\in C(X,Y);~ \sup_{x\in X}d(f(x),g(x))<\epsilon\}
\end{equation*}
where $f$ ranges over $C(X,Y)$, and $\epsilon$ over all the real numbers.
The topology of uniform convergence coincides with the compact open topology whenever the domains $X$ is compact \cite[Chapter 7, Theorem 11]{kelley_general_1975}.
For noncompact $X$ the latter is coarsest than the former.
\begin{theorem}[Myers \cite{myers_equicontinuous_1946}]\label{genAA}
    Let $C$ be the family of all continuous functions from a regular, locally compact, and Hausdorff topological space into a metric space, equipped with the compact-open topology.
    Then a subfamily $T$ of $C$ is compact if and only if
    \begin{itemize}
        \item[(1)] $T$ is closed in $C$,
        \item[(2)] $T[x] := \{ f(x); f \in T\}$ has compact closure for each $x\in X$, and
        \item[(3)] the family $T$ is equicontinuous.
    \end{itemize}
\end{theorem}
\par Comparing Theorem \ref{genAA} with the result originally proved by Myers in \cite{myers_equicontinuous_1946}, the hypothesis that $\overline{T[x]}$ is compact for all $x\in X$ replaces the connectedness of $X$ and completeness of $Y$; see also \cite{kelley_general_1975} pages 233-234.
\begin{theorem}[Arens \cite{arens_topology_1946}]\label{thm_Arens}
    The group of all homeomorphisms of a compact Hausdorff space is a topological group when equipped with the compact-open topology.
\end{theorem}
\begin{lemma}\label{con_seq_id}\rm{
    Let $G$ be a compact topological group.
    Given $g\in G$, there exists a sequence of integers $n_k\to \infty$ such that $g^{n_k}\to e$}.
\end{lemma}
\begin{proof}
    Since $G$ is compact, given $g\in G$, the sequence $(g^n)_{n\in\mathbb{N}}$ has a convergent sub-sequence  $(g^{m_s})$. 
    Given any integer $k$, let $m_r\in \{m_s\}$ such that $n_k:=m_r-m_k>k$, $m_k\in\{m_s\}$.
    Using the continuity of the product and inversion operations of $G$, and letting $k\to \infty$, one obtains
    \begin{equation*}
        g^{n_k}=g^{m_r}g^{-m_k} \to e,
    \end{equation*}
    as desired.
    \end{proof}
\begin{theorem}\label{equi_flow_recur}
    An equicontinuous flow on a compact metric space is recurrent.
\end{theorem}
\begin{proof}
   It the flow is periodic, there is nothing to be done.
    Suppose that the flow is not periodic. 
    Then $G=\{t\}_{t\in \mathbb{R}}$ is an equicontinuous group of homeomorphisms of $X$.
    It follows from Theorems \ref{genAA} and \ref{thm_Arens} that the closure $\overline{G}$ of $G$ in the compact open topology is a compact topological group.
    Hence, from Lemma \ref{con_seq_id}, given $t\in\mathbb{R}$, there exists a sequence $n_k\to\infty$, $n_k\in\mathbb{Z}$, such that $ t_{n_k}\to \mathrm{id}$ (with respect to the compact open topology).
    Since $X$ is compact, this implies the same sequence converges to the identity also in the uniform topology.
    In other words, $t_{n_k}\to \mathrm{id}$ uniformly, hence the flow is recurrent.
\end{proof}

\subsection{Asymptotic cycles on smooth manifolds}
\par Let $M$ be a compact manifold, and $C=C^\infty(M,\mathbb{S}^1)$ be the group of all smooth functions $f:M\to \mathbb{S}^1$, with the usual pointwise multiplication.
We consider the following subgroup of $C$  
\begin{equation*}
    R := \{f\in C; f(x)=\exp{(2\pi i H(x))}, H:M\to \mathbb{R} \text{ is smooth}\}.
\end{equation*}
Every element $f\in C$ is given locally by $f(x)=\exp 2\pi i H(x)$, where $H(x)$ is some smooth real function determined up to an additive constant.
If $H_1$ and $H_2$ are determined as above in two different coordinate systems, $dH_1$ and $dH_2$ agree in the overlapping region.
Then, for each $f\in C$, there is a closed 1-form $\omega_f$ given locally by $dH$.
Consider the local parametrization of the circle given by $\theta(\exp{(2\pi it)})=t$, and the usual volume form on $\mathbb{S}^1$ described locally by $d\theta$.
Let $\alpha:I\to M$ be a smooth curve with  $X(x)=\partial_t|_0\alpha(t)$.
If $f=\exp 2\pi i H(x)$ in a neighborhood of $\alpha(0)$, then it is easily seen that 
\begin{align*}
    f^*d\theta(X)=\omega_f(X).
\end{align*}
Therefore $\omega_f=f^*d\theta$.
In fact, the association $f\mapsto \omega_f$ give us  
\begin{equation*}
    [f]\in C/R \mapsto [\omega_f]\in H^1_{dR}(M;\mathbb{Z}),
\end{equation*}
an isomorphism. 
In what follow we consider a fixed flow $\Phi $ on $M$ induced by a smooth vector field $X$. 
\begin{definition}[Quasi-regular points]
    A point $x\in M$ is called \emph{quasi-regular} provided  
    \begin{equation*}
        \lim_{T\to\infty}1/T\int_0^Tf(xt)dt
    \end{equation*}
    exists for every real-valued continuous function defined on $M$.
    \end{definition}
\begin{remark}
    The symbol $f^{'}(x)$  will  denote the derivative of $f$ in the flow direction, that is, 
    \begin{align}\label{f linha}
        f^{'}(x)={\partial_t}|_{0}f(xt).
    \end{align}
\end{remark}
\begin{definition}
    Let $x$ be a quasi-regular point.
    Consider the functional $A_x:H^1(M;\mathbb{Z})\to \mathbb{R}$ given by
    \begin{equation*}
        A_x[ f]= \lim_{T\to \infty}\frac{1}{2\pi i T}\int_{0}^{T}  \frac{f^{'}(xt)}{f(xt)}dt.
    \end{equation*}
    The \emph{asymptotic cycle} associated with $x$ is the linear extension of $A_x$ to $H^1(M;\mathbb{R})$. 
\end{definition}
\begin{definition}
    A finite Borel measure $\mu$ on $M$ is said to be \emph{invariant} under the flow  provided 
    \begin{equation*}
    \int\displaylimits_{M} (f(xt)-f(x))d\mu(x)=0
    \end{equation*}
    for all continuous functions $f:M\to \mathbb{R}$ and all $t\in \mathbb{R}$.
\end{definition}
\begin{definition}
    The asymptotic cycle associated with an invariant measure $\mu$ is the homological class determined by the continuous current\footnote{
        See \cite[\S 18]{de_rham_varietes_1973} for more details on currents and how they can be seen as homology classes.
    } 
    $A_\mu:\mathcal{D}_1\to \mathbb{R}$, given by
    \begin{equation*}
        A_\mu(\omega)=\int\displaylimits_{M}\omega(X)d\mu.
    \end{equation*}
\end{definition}
It is easily seen that $\partial A_\mu=0$ (that is, $\partial A_\mu$ is a cycle), hence it determines an integral homological class that will be denoted too by $A_\mu$.
\begin{lemma}
    Let $\mu$ be an invariant measure under a flow.
    Consider $A^{'}_\mu:C\to \mathbb{R}$ given by
    \begin{equation*}
        A^{'}_\mu([f])=\int\displaylimits_{M}A_x[f]d\mu(x).
    \end{equation*}
    Then $A_\mu$ and the natural linear extension of $A^{'}_\mu$ to $H_{dR}^1(M;\mathbb{R})$ coincide. 
\end{lemma}
\begin{proof}
    By \cite[Theorem on page 274]{schwartzman_asymptotic_1957}, we have 
    \begin{equation*}
        A^{'}_\mu([f])=\frac{1}{2\pi i }\int\displaylimits_{M} \frac{f^{'}(x)}{f(x)}d\mu.
    \end{equation*}
    Expressing $f$ locally by $f(x)=\exp 2\pi iH(x)$, we have 
    \begin{equation*}
        f^{'}(x)=2\pi iH_*(\partial_t|_0(xt))f(x).
    \end{equation*}
    Hence 
    \begin{equation*}
        f^{'}(x)/(2\pi i f(x))=\partial_t|_0H(xt)=\omega_f(X(x)).
    \end{equation*}
    It follows that
    \begin{equation*}
        A^{'}_\mu([f])= \int\displaylimits_{M} \omega_f(X(x))d\mu=A_\mu(\omega_f).
    \end{equation*}
\end{proof}
\begin{remark}
    Given a quasi-regular point $x$, there exists a unique measure $\mu_x$ such that 
    \begin{equation*}
        \lim_{T\to\infty}1/T\int_0^Tf(xt)dt=\int\displaylimits_{M}f(y)d\mu_x(y)
    \end{equation*}
    for every continuous function $f$ \cite[\S 2]{schwartzman_asymptotic_1957}.
    It follows that 
    \begin{equation*}
        A_x(f)=\lim_{T\to \infty}\frac{1}{2\pi iT}\int_0^T\frac{f^{'}(xt)}{f(xt)}dt=\lim_{T\to \infty}\frac{1}{ T}\int_0^T\omega_f(X)dt=
    \end{equation*}
    \begin{equation*}
        \int\displaylimits_{M}\omega_f(X)(y)d_{\mu_x}(y)=A_{\mu_x}(\omega_f),    
    \end{equation*}
    hence $A_x=A_{\mu_x}$.
\end{remark}
\begin{lemma}\label{canonicalCycleVolPreservingFlow}
    Let $M$ be a closed orientable manifold. 
    Let $\Omega$ be a volume form on $M$ and $X$ be a smooth vector field on $M$ preserving $\Omega$.
    Consider the differential form $\omega=\iota_X\Omega$ and the measure $\mu$ given by $\mu(B)=\int_B\Omega$ on the open sets $B$ of $M$. 
    Then $\mu$ is invariant under the flow determined by $X$ and $A_\mu=[\omega]$%  
    \footnote{
        Here, $\omega$ is seen as a \emph{diffuse current} $\omega:\mathcal{D}_{n-1}\to\mathbb{R}$, given by $\omega(\eta)=\int\displaylimits_{M}\omega\wedge \eta$.
        }.
    In particular, if $\omega$ determines a nonzero cohomology class, then $A_\mu\neq 0$ in $H_1^{dR}(M)$.
\end{lemma}
\begin{proof}
    Let  $\eta$  be a 1-form on $M$.
    Using the identity  $\iota_X(\eta\wedge\Omega)=\iota_X\eta\wedge\Omega+\eta\wedge \iota_X\Omega$ and the definition of $\mu$, one has
    \begin{align*}
        \int\displaylimits_{M}\eta(X)d\mu &=\int\displaylimits_{M}\eta(X)\Omega\\
        &=\int\displaylimits_{M}\iota_X\eta\wedge\Omega \\
        &=\int\displaylimits_{M}(\iota_X(\eta\wedge\Omega)+\eta\wedge \iota_X\Omega)\\
        &=\int\displaylimits_{M}\eta\wedge\omega.
    \end{align*}
    Thus, since $d(f\omega)=df\wedge\omega-fd\omega$ and $d\omega=0$, it follows from Stokes's theorem that  
    \begin{equation*}
        \int\displaylimits_{M}df(X)d\mu=\int\displaylimits_{M}df\wedge\omega=\int\displaylimits_{M}d(f\omega)=0
    \end{equation*}
    for all $f\in C^\infty(M)$.
    Using this, we will prove that $\mu$ is invariant.
    Let $t\in \mathbb{R}$ and $f\in C^\infty(M)$.
    By Fubini's theorem
    \begin{align*}
        \int\displaylimits_{M}(f(xt)-f(x))d\mu(x)&=\int\displaylimits_{M}\left(\int_0^tdf(X(sx))ds\right)d\mu(x) \\
        &=\int_0^t\left(\int\displaylimits_{M}df(X(sx))d\mu(x)\right)ds \\
        &=0.
    \end{align*}
    Now, since $\int\displaylimits_{M}(f(xt)-f(x))d\mu(x)=0$ for every $t\in \mathbb{R}$ and $f\in C^\infty(M)$, it follows that $\mu$ is invariant.
    Therefore, $A_\mu$ is well defined and 
    \begin{equation*}        A_\mu(\eta)=\int\displaylimits_{M}\eta(X)d\mu=\int\displaylimits_{M}\eta\wedge\omega =\omega(\eta),
    \end{equation*}
    where $\omega$ is seen as a diffuse current.
    Now, by De Rham's theory of currents, the form $\omega$ determines a nonzero cohomology class if and only if, as a diffuse current, it determines a nonzero homological class. %
\end{proof}
\begin{lemma}\label{asymptotic cycles of periodic points}
    Let $M$ be a manifold, $\Phi$ be a continuous flow on $M$, and $x \in M$ a periodic point of the flow
    \footnote{
        The period function, $\lambda:M\to\mathbb{R}\cup\{\infty\}$, is defined by $\lambda(x)=\inf_{t\geq 0}\{ xt=x\}$. 
        A point is said to be periodic provided $\lambda(x)\neq \infty$.
    }.
    Then $x$ is quasi-regular and $\lambda(x) A_x=[C_x]$, where $[C_x]$ represents the integral homological class determined by the orbit of $x$ under the flow, and $\lambda$ is the period function.
\end{lemma}
\begin{proof}
    If $x$ is a fixed point, the assertion follows.
    Suppose  $x$ is not a fixed point.
    For each smooth function $f:M\to \mathbb{R}$, it is easily seen that
    \begin{equation*}
        \lim_{t\to\infty}1/T\int_0^Tf(xt)dt=1/\lambda(x)\int_0^{\lambda(x)} f(xt)dt.
    \end{equation*}
    Therefore every periodic point $x\in M$ is quasi-regular.
    Set $\alpha:[0,T]\to M$ given by $\alpha(t)= xt$. 
    Let $f\in C$.
    Then
    \begin{align*}
        A_x[f] &= \lim_{T\to \infty}\frac{1}{2\pi i T}\int_{0}^{T}  \frac{f^{'}(xt)}{f(xt)} \\
        &= \lim_{T\to \infty}\frac{1}{2\pi i T}\int_{0}^{T}  2\pi i\left(\omega_f(\alpha^{'}(t))dt\right) \\
	& =\lim_{T\to \infty}\frac{1}{ T}\int\displaylimits_{\alpha([0,T])}\!\!\!\!\!\!\omega_f \\
	& =1/\lambda(x)\int\displaylimits_{\alpha [0,\lambda(x)]}\!\!\!\!\!\!\omega_f.
    \end{align*}        
    Parametrizing the orbit of $x$ by $\alpha$, and denoting by $[C_x]$ the integral homological class determined by the $x$-orbit, we have
	\begin{equation*}
            [C_x]([f])= \int\displaylimits_{\alpha[0,\lambda(x)]}\!\!\!\!\!\!\omega_f=\lambda(x) A_x[f].
	\end{equation*}
    Therefore, 
	\begin{equation*}
            \lambda(x) A_x= [C_x].
        \end{equation*}
\end{proof} 
\begin{theorem}\label{Schwartzman}
    A smooth flow  on a closed manifold $M$ admits a cross-section if, and only if, for every invariant measure $\mu$ the homological class determined by $A_\mu$ is nonzero.
\end{theorem}
\begin{proof}[Proof's sketch] 
    The set $\mathcal{C}=\{A_\mu; \mu \text{ invariant measure}\}$ is convex and weakly compact. 
    Since $\mathcal{B}_1$ (the boundary subspace) is  weakly-closed and $\mathcal{C}\cap \mathcal{B}_1=\emptyset$, by Hahn-Banach Theorem there exists a continuous linear functional $\varphi: \mathcal{D}_1^{'}\to\mathbb{R}$ such that 
    \begin{equation*}
        \varphi\vert_{\mathcal{C}}>0\text{ and }\phi|_{\mathcal{B}_1}=0.
    \end{equation*}
    Due to a famous theorem of Laurent Schwartz \cite{schwartzman_asymptotic_1957}, we know that the space of currents is reflexive.
    Hence, there exists an 1-form $\eta$ such that $\varphi=\eta$. 
    Therefore
    \begin{equation*}
        d\eta(\phi)=\eta(\partial\phi)=\varphi(\partial\phi)=0,
    \end{equation*}
    concluding that $\eta $ is closed.
    The positivity condition of $\varphi$ on $\mathcal{C}$ implies that   
    \begin{equation*}
        A_\mu(\eta)=\eta(A_\mu)=\varphi(A_\mu)>0
    \end{equation*}
    for every invariant measure $\mu$.
    Now, the set of all invariant measures is weakly compact.
    By  Tischler's argument \cite{tischler_fibering_1970}, $\eta$ can be approximated in $\mathcal{D}_1$ by forms with integral periods (we need a small change in the argument to obtain an approximation in $\mathcal{D}_1)$.
    It follows that there exists a form $\omega$ with integral periods such that $\omega(A_\mu)>0$ for all invariant measure $\mu$.
    Since $\omega=f^*d\theta$ for some smooth function $f:M\to \mathbb{S}$, it follows from \cite[Theorem on page 281]{schwartzman_asymptotic_1957} that the flow admits a global cross-section.
\end{proof}

\section{On $k$-forms of rank $k$}\label{sec:k-forms}
In this section we prove some preliminary (and also more general) results that will help us in the proof of Theorem \ref{pre_char} in the next section.
% Let $\omega$ be a $k$-form on a $n$-dimensional manifold $M$.
% We will denote by $\ker\omega$ the distribution on $M$ given by
% \begin{equation*}
%     x\rightarrow \ker\omega\rvert_x=\{v\in T_xM; \iota_v\omega=0\}.
% \end{equation*}
% The \emph{rank} of $\omega$ at a point $x\in M$ is defined by $r_x=n-\dim\ker_x$.
% A form $\omega$ is said to be of \emph{constant rank} if $r_x$ is independent of $x$. 
% Let $\omega$ be a closed form of rank $k$.
% Given $X,Y\in\ker\omega$, we use the identity $\iota_{[X,Y]}=[\mathcal{L}_X,\iota_Y]$ and obtain
% \begin{equation*}
%     \iota_{[X,Y]}\omega=\mathcal{L}_X \iota_Y\omega - \iota_Y\mathcal{L}_X\omega=\mathcal{L}_X \underbrace{\iota_Y\omega}_{0} - \iota_Y(\iota_X \underbrace{d\omega}_0 + d\underbrace{\iota_X\omega}_0) = 0,
% \end{equation*}
% that is, the distribution $\ker\omega$ is involutive.
% It follows from Frobenius' Theorem that $\ker\omega$ is tangent to a foliation $\mathcal{F} $ of codimension $k$. In the case $k=n-1$ the foliation give us a volume-preserving flow (cf. Section \ref{sec:Prelim}).
\begin{theorem}\label{compFolClosedForm} 
    Let $M$ be a manifold with a closed $k$-form $\omega $ of rank $k$.
    Then $\omega$ is intrinsically harmonic if, and only if, there exists a closed form $\eta$ of rank $q=n-k$ such that $\ker\omega  \cap \ker\eta  = \{0\}$.
\end{theorem}
\begin{proof} 
    First, let $\omega$ be a closed $k$-form of rank $k$.
    Suppose $g$ is a Riemannian metric on $M$ with respect to which $\omega$ is harmonic.
    Then the form  $\eta := *\omega$ is a closed $q$-form with the desired property.
    \par Conversely, let $\eta$ be a $q$-form whose kernel is transverse to $\ker\omega$.
    Choose Riemannian metrics $g_1,\ g_2$ on the bundles $\ker\omega$\ and $\ker\eta$ and set 
    \begin{equation*}
        g = g_1 \oplus g_2.
    \end{equation*}
    The forms $*\omega$ and $\eta$ are volume forms in $\ker\eta^{\perp}$ and have the same nullity space.
    Hence
    \begin{equation*}
        *\omega = s\eta
    \end{equation*}
    where $s: M \to \mathbb{R}$ is a smooth function.
    We have
    \begin{equation*}
        \omega \wedge *\omega =  \lvert\omega\rvert^2 \Omega \ \text{  and  } \ \omega \wedge \eta = r\Omega,
    \end{equation*}
    where $\Omega$ is the volume form of $g$, and $r: M \to \mathbb{R}$\ is a smooth function.
    It follows that
    \begin{equation*}
        \lvert\omega\rvert^2\Omega = \omega \wedge *\omega =\omega \wedge s \eta = sr\Omega.
    \end{equation*}
    Therefore $sr = \lvert\omega\rvert^2$.
    Let $\{E_i\} $ and $\{ F_j\}, i = 1, \cdots q,\ j =1, \cdots\!, k$\ be orthonormal bases for $\ker\omega,\ \ker\eta$, respectively, and let $f: M \to \mathbb{R}$ be a positive function.
    Setting $E'_i= f^{-\frac{1}{2}}E_i$, then $\{E'_i, F_j\}$\ is an orthonormal basis with respect to the metric tensor
    \begin{equation*}
        g_f = fg_i \oplus g_2.
    \end{equation*}
    Finally, it remains to find a function $f$ satisfying
    \begin{equation}\label{semiconformal}
        *_{g_f}\omega = \eta.
    \end{equation}
    Indeed, if $f$\ is such a solution,\ $\omega$\ is harmonic with respect to $g_f$,\ since $\eta$\ is closed, and we have:
    \begin{equation*}
        \omega(F_1, \cdots\!, F_k) = *_{g_f}\omega(E'_1, \cdots\!, E'_q) = \eta(E'_1, \cdots\!, E'_k) = f^{\frac{p-n}{2}}\eta(E_1, \cdots\!, E_k).
    \end{equation*}
    Thus
    \begin{equation*}
        f = (\omega(F_, \cdots\!, F_k)/\eta(E_1, \cdots\!, E_k))^{\frac{2}{k-n}} = s^{\frac{2}{k-n}}
    \end{equation*}
    is a solution of (\ref{semiconformal}), and the proof is complete.
\end{proof}
\begin{definition}
    A $k$-form $\omega$ on $M$ is said to be \emph{transitive} at $x\in M$ if $x$ is contained in a closed embedded smooth $k$-dimensional submanifold $\Sigma \subset M$, to which $\omega$ restricts to a volume form \footnote{
        In other words, the restriction $\omega\rvert_\Sigma$ is a top degree form without zeros. We usually denote this simply by writing $\omega\rvert_\Sigma>0$.
    }.
    The form $\omega$ is said to be transitive if it is transitive at every $x\in M$ where it does not vanish.
\end{definition}
\begin{lemma}(Calabi's argument)\label{Calabi's argument}
    Let $M$ be a manifold and $\omega$ be a nowhere-vanishing $k$-form on $M$.
    Suppose that every $x\in M$ for which  $\omega\rvert_x\neq 0$ is contained in a closed $k$-dimensional submanifold $\Sigma$ with normal trivial bundle and such that $\omega\rvert_\Sigma>0$.
    Then there exists a closed $(n-k)$-form $\psi$ such that $\omega\wedge\psi>0$ everywhere.
\end{lemma}
\begin{proof}
    We generalize an argument used in the case of $1$-forms by Calabi in \cite{calabi_intrinsic_1969}.
    Let $x\in M$ be contained in a closed submanifold $\Sigma$, with normal trivial bundle.
    Let 
    \begin{equation*}
        \phi:\Sigma\times  \mathbb{D}^{n-k}\to V
    \end{equation*}
    be a diffeomorphism satisfying $\phi(x,0)=x$ for $x\in \Sigma$ (this can be obtained by trivializing the normal bundle of $\Sigma$).
    Writing $\phi=(x,y)$, we have 
    \begin{equation*}
        \omega\rvert_x(\partial x_1,\cdots\!,\partial x_p)>0
    \end{equation*}
    for $x\in \Sigma$.
    By compactness,  $\omega\wedge dy>0$ in an open set $V_0$ with $\Sigma\subset V_0\subset V$.
    Let $\epsilon>0$ be such that $V_{2\epsilon}\subset V_0$ and $h:[0,2\epsilon]\to\mathbb{R}$ be a smooth function satisfying
    \begin{equation*}
    \left\{\begin{array}{lll}
        h(t)=1 \text{ if }t<\epsilon\\
        h(t)=0 \text{ if }t\geq2\epsilon\\
        h\geq 0 \text{ everwhere}.
    \end{array}\right.
    \end{equation*}
    The differential form
    \begin{equation*}
        \eta=h(\lvert y\rvert^2)dy
    \end{equation*}
    is closed on $M$ and satisfies $\omega\wedge\eta \geq 0$, $\omega\wedge\eta>0$, where $\eta\neq 0$ and $\omega\wedge\eta=\omega\wedge dy>0$ in $V_\epsilon$.
    Consider an countable locally finite cover of $M$ by opens sets $\{V_i\}_{i\in I}$, with respective closed forms $\eta_i$ satisfying the three conditions above.
    The expression
    \begin{equation*}
        \psi=\sum_{i\in I}\eta_i	
    \end{equation*} 
    gives a well-defined closed differential form satisfying $\omega\wedge\eta>0$ in $M$. 
\end{proof}

\subsection{Intrinsically harmonic nowhere-vanishing $(n-1)$-forms}\label{sec:(n-1)-forms}
 In \cite{honda_harmonic_1996}, Honda proves that if a transitive $(n-1)$-form whose every zero is isolated is ``well-behaved'' around its zeros, then it is intrinsically harmonic. 
 Honda results make it somewhat natural to expect that a nowhere-vanishing $(n-1)$-form supporting a global cross-section must be intrinsically harmonic.
 In this section we show that this is in fact the case, effectively proving Theorem \ref{pre_char}.
 We start with the following result.
\begin{lemma}\label{trnsNwVanIH}
    Let $M$ be an orientable manifold. Suppose that $\omega$ is a nowhere-vanishing transitive closed 1-form (or $(n-1)$-form) on $M$. Then $\omega$ is intrinsically harmonic
\end{lemma}
\begin{proof}
    Since $M$ is an orientable manifold, any closed 1-dimensional or any closed orientable $(n-1)$-dimensional submanifold of $M$ has a trivial normal bundle in $M$.
    Thus, in the case that $\omega$ is a transitive 1-form or a transitive $(n-1)$-form, we can apply Lemma \ref{Calabi's argument} to obtain a closed form $\psi$ defined on $M$ such that $\omega\wedge\psi>0$.
    The proof is completed by applying Theorem \ref{compFolClosedForm} and observing that, in both cases, $\omega$ and $\psi$ have the correct rank.  
\end{proof}
% The following theorems establish two criteria to decide when a closed nowhere-vanishing $(n-1)$-form is  intrinsically harmonic.
%It is made concerning the existence of a global cross-section or a complementary foliation to the induced flow given in the Lemma.
%  Necessary and sufficient conditions for the existence of a global cross-section for a flow given in  \cite{fried1982geometry} and \cite{schwartzman1957asymptotic}, for example. 
%\begin{theorem}\label{An intrinsic characterization of nowhere-vanishing harmonic (n-1)-forms}
%Let $M$ be a closed orientable manifold.
% A nowhere-vanishing volume-preserving flow defined on $M$ admits global cross-section if and only if the induce closed nowhere-vanishing $(n-1)$-form is intrinsically harmonic
%\end{theorem} 
We are now able to prove Theorem \ref{pre_char}.
\begin{proof}[Proof of Theorem \ref{pre_char}]
    Let $\omega$ be a closed nowhere-vanishing harmonic $(n-1)$-form defined on an orientable Riemaniann manifold $(M,g)$.
    Then $\omega\wedge *_g\omega$ is  a volume form on $M$, with $*_g\omega$ being a closed nowhere-vanishing 1-form.
    By Tischler's argument \cite{tischler_fibering_1970}, there is a 1-form $\eta$ with integral periods such that $\omega\wedge\eta$ is a volume form on $M$.
    Each of such $\eta$ can be written as $f^*d\theta$ for some smooth function $f:M\to\mathbb{S}^1$, where $d\theta$ is the usual volume form on $\mathbb{S}^1$.
    Since $\eta=f^*d\theta$ is nowhere-vanishing, the function $f$ is a submersion.
    Due to Ehresmann's Lemma \cite{ehresmann1947topologie,earle1967foliations}, $\mathcal{B}=\{M,f,\mathbb{S}^1, F\}$ defines a fiber bundle with fiber $F$ diffeomorphic to $f^{-1}(\theta)$, $\theta\in\mathbb{S}^1$.
    Since the flow  $\Phi$ determined by $\omega$ induces a 1-dimensional foliation transversal to the fibers of this bundle, it follows from Ehresmann's Theorem, \cite[Proposition 1, pag. 91]{camacho2013geometric}) that 
    $\mathcal{B}=\{M,f,\mathbb{S}^1\}$ is a foliated bundle, the foliation given by the orbits of $\Phi$.
    In particular, every fiber of this bundle intercepts every orbit of $\Phi$.
    In other words, $f^{-1}(\theta)$ is a global cross-section to $\Phi$ for all $\theta\in\mathbb{S}^1$ (not necessarily connected).
    \par Conversely, suppose that the flow $\Phi$  induced by $\omega$ admits a global cross-section $\Sigma$.
    We can assume $\omega\rvert_\Sigma>0$.
    Since we can slide $\Sigma$ transversely in the flow direction, it follows that $\Phi$ has a global cross-section at every point $x\in M$, and hence $\omega$ is transitive.
    Finally, applying Lemma \ref{trnsNwVanIH}, we conclude that $\omega$ is intrinsically harmonic. 
\end{proof} 

\section{A characterization of volume-preserving flows with a global cross-section}\label{sec:vol_pre_char}
\par If a closed manifold supports a smooth flow with a global cross-section, then there exists a Riemannian metric so that each orbit under the flow is a geodesic.
In this section, we provide the converse for volume-preserving Riemannian flows (cf. Theorem \ref{int_har_char2}).
What is more, the existence off the global cross-section constructed by our argument depends only on the cohomological class determined by a $(n-1)$-form, which is completely determined by the volume form and the volume-preserving flow.
\par We begin by recalling some definitions and a preliminary result due to Carri{\`e}re.
\begin{definition}Let $M$ be a manifold.
\begin{itemize}
    \item[--] A flow on $M$ is said to be \emph{geodesible} if there exists a Riemannian metric $g$ on $M$ such that every orbit under the flow is a geodesic of $g$.
    \item[--] Given a Riemannian metric $g$ on $M$, a flow on $M$ is said to be \emph{isometric} if, for all $t\in\mathbb{R}$, the function $t:M\to M$ is an isometry. 
    \item[--] A vector field $X$ on $(M,g)$ is said to be \emph{Killing} if the flow generated by $X$ is isometric. 
    \item[--] A flow on $M$ is said to be \emph{isometrizable} if there exists a Riemannian metric $g$ on $M$ such that the flow is given by a nowhere-vanishing Killing field.
\end{itemize}
\end{definition}

\begin{theorem}[Carri{\`e}re, \cite{carriere_flots_1984}]\label{thmCar}
    A nowhere-vanishing Riemannian vector field is geodesible if and only if it is isometrizable.
\end{theorem}
\begin{theorem}\label{int_har_char2}
    Let $\omega$ be a nowhere-vanishing closed $(n-1)$-form on an orientable closed manifold $M$, and $X$ be the unique volume preserving field generating the foliation $\mathcal{F} $ induced by $\omega$.
    Suppose $[\omega]\neq 0$ in $H^{n-1}_{dR}(M)$.
    The following are equivalent.
    \begin{itemize}
        \item[(1)] $X$ is  Riemannian and geodesible;
        \item[(2)] $X$ is Killing for some Riemannian metric on $M$;
        \item[(3)] $\omega$ is intrinsically harmonic. 
    \end{itemize}
\end{theorem}
\begin{proof}
    Note that the equivalence between assertions (1) and (2) is simply Carri{\`e}re's result (Theorem \ref{thmCar}).
    Suppose now that (2) holds, so that there is on $M$ a metric tensor $g$ inducing a metric $d$ with respect to which the flow induced by $X$ is equicontinuous.
    Now, Theorem \ref{equi_flow_recur} then ensures that flow is a recurrent flow. 
    We know, due to a result of Schwartzman's \cite[§ 6]{schwartzman_asymptotic_1957}, that a recurrent flow has only one asymptotic cycle.
    Since the class of the closed $(n-1)$-form $\omega:= \iota_X\Omega$ induced by $X$ is an asymptotic cycle (cf. Lemma \ref{canonicalCycleVolPreservingFlow}), and since $[\omega] \neq 0$ by assumption, it follows from Schwartzman's results that the flow admits a global cross-section.
    Finally, applying Theorem \ref{pre_char} we can conclude that (2) implies (3).
    \par Conversely, suppose now that $\omega$ is intrinsically harmonic.
    Let $h$ be a Riemannian metric on $M$ such that $\eta=*_h\omega$ is a closed form.
    Let $g$ be a Riemannian metric on $M$ such that $g(X,X)=1$ and $\ker\eta$ is orthogonal to $X$.
    The $1$-form $\eta$ is the characteristic form (with regard to $g$) of the foliation $\mathcal{F} $.
    It then follows from Rummler's criterion (cf. \cite[Proposition 1]{rummler_quelques_1979}) that the leaves of $\mathcal{F} $ are \emph{minimal submanifolds} of $(M,g)$.
    It is well-known that a minimal $1$-dimensional manifold becomes a geodesic when properly parameterized (see the proof Theorem \ref{riem_geod_flow} below for a more detailed account of this fact). 
    Hence the flow lines of $X$ are traces of geodesics, that is, $X$ is geodesible.
    \par Suppose further that the generator field $X$ is normalized to satisfy $\eta(X)=1$. 
    Then 
    \begin{equation*}
        \mathcal{L}_X\eta=\iota_X\underbrace{(\mathrm{d}\eta)}_{0}+\mathrm{d}\underbrace{(\iota_X\eta)}_{1}=0,
    \end{equation*}
    and therefore the flow of $X$ preserve the leaves of the codimension $1$ foliation $\mathcal{E}$ induced by $\ker\eta$, that is, it carries leaves into leaves. 
    Now, by construction, the foliation $\mathcal{F} $ (that is, the unparameterized flow of $X$) is orthogonal to the foliation $\mathcal{E}$.
    Using foliated charts, we can find a covering of $M$ by open sets $U_\alpha$ with the following property: for each $\alpha$ there exists an interval $I_\alpha$ containing $0$, and a plaque of the foliation $\mathcal{F} $ contained in $U_\alpha$, such that the mapping 
    \begin{align*}
        \Psi:I_\alpha\times P_\alpha &\to U_\alpha \\
        (t,x) &\mapsto xt
    \end{align*}
    is a diffeomorphism.
    For any product metric on $I_\alpha\times P_\alpha$, the foliation induced by $\partial_t$ is Riemannian.
    We fix such a metric, say $h$, and let $g_\alpha :=(\Psi^{-1})^*h$.
    Since $\Psi_*(\partial_t)=X$, it follows that $\Phi$ is Riemannian with respect to $g_\alpha$.
    Finally, we consider a partition of the unity $\{\lambda_\alpha\}$ subordinate to $\{U_\alpha\}$ and set 
    \begin{equation*}
        g :=\sum\lambda_\alpha g_\alpha.
    \end{equation*}
    With respect to $g$, the flow $\Phi$ is Riemannian.
    Thus $X$ is Riemannian and geodesible, and we see that (3) implies (1), completing the proof. 
    \end{proof}
\begin{remark}\label{counterexamples}
    Equicontinuity plays an important role in the proof of Theorem \ref{int_har_char2}. 
    One could ask if the properties of being geodesible or volume-preserving could, alone, imply the desired equicontinuity. 
    The answer is \emph{no}, and the counterexample comes from expansive geodesic flows. 
    \par Informally, an expansive flow is one for which distinct orbits grow apart.
    To formalize this, we ask that, given distinct points $x,y$, any reparameterization of the $y$ orbit always ends up leaving a transverse ball of fixed radius around the orbit of $x$ after finite time (cf. \cite{paternain_expansive_1993} for a precise definition).
    \par It is not hard to see that an expansive flow can never be equicontinuous.
    It is known, due to the work of Paternain \cite[Corollary 2]{paternain_expansive_1993}, that if a $2$-dimensional closed Riemannian manifold has no conjugated points and its Riemannian covering has no bi-asymptotic geodesics, then its geodesic flow is expansive.
    On the other hand, the \emph{geodesic flow is geodesible}: its flow lines are exactly the geodesics of the Sasaki metric on the tangent bundle.
    Thus, this entire class of $2$-manifolds provides examples of geodesible flows that are not equicontinuous.
    The same geodesic flows provide counterexamples to the volume-preserving case since the geodesic flow preserves the Riemannian volume of the corresponding metric \cite[Chapter III, Exercise 14]{do_carmo_riemannian_1992}. 
    \par Yet another class of examples of geodesible flows which are not recurrent is given by pointwise periodic flows which are not globally periodic but have a bounded period function.
    These are geodesible, due to a result of Wedsley's \cite{wadsley1975geodesic},  but can't be equicontinuous as a consequence of Theorem 1 in \cite{francca2023pointwise}.
\end{remark}

\section{A characterization of the $n$-torus}\label{sec:tor_char}
\par Let $M$ be a closed $n$-dimensional manifold supporting a set of $n-1$ linearly independent (in particular, nonvanishing) closed $1$-forms $\lambda_i$.
We suppose as well that 
\begin{equation*}
    \omega := \lambda_1 \wedge \cdots \wedge \lambda_{n-1}
\end{equation*}
determines a nonzero cohomological class. 
We use the same notation as in Section \ref{rank1_form}.
In particular, $M$ is orientable, has a volume form $\Omega$, and $X_0$ and $\lambda_0$ are, respectively, the unique vector field and $1$-form induced on $M$ by $\omega$ satisfying the relations $\lambda_0\wedge\omega = \Omega$ and $\iota_{X_0}\Omega = \omega$.
Moreover, $\{X_0, \cdots\!, X_{n-1}\}$ is a global frame satisfying $\lambda_i(X_j) = \delta_{ij}$.
\begin{theorem}\label{riem_geod_flow}
    $M$ supports a Riemannian metric tensor $g$ with respect to which $X_0$ is geodesible and Riemannian. 
\end{theorem}
\begin{proof}
    As noted in Remark \ref{lambda0notclosed}, each one of the fields $X_i$ is foliate.
    This means, in other worlds, that there are $n-1$ functions $f_i: M \to \mathbb{R}$ such that $[X_0, X_i] = f_iX_0$, for $i = 1, \cdots\!, n-1$.
    We define $f_0 \equiv 0$ so that $[X_0, X_0] = f_0X_0$ as well.
    \par We now define on $M$ a metric tensor $g$ by imposing orthonormality to the global frame $\{X_0, X_1, \cdots\!, X_{n-1}\}$, and extending it linearly to the whole tangent bundle.
    Another way to describe $g$ is 
    \begin{equation*}
        g = \sum_{i=0}^{n-1}\lambda_i\otimes\lambda_i.
    \end{equation*}
    So that each $1$-form is the dual of the corresponding vector field $X_i$ with respect to $g$.
    The metric $g$ is thus as smooth as the fields $X_i$, and using this metric the normal bundle of $\mathcal{F} $, that is, the distribution 
    \begin{equation*}
        Q_x := ~^{\textstyle \mathrm{T}_xM}\!\big/_{\textstyle \mathrm{T}_x\mathcal{F} }
    \end{equation*}
    can be identified with the normal bundle of $X_0$, that is, the distribution 
    \begin{equation*}
        \bigoplus_{i=1}^{n-1}\mathbb{R}X_i.
    \end{equation*}
    Moreover, for $i, j \in \{1, \cdots\!, n-1\}$, the metric tensor $g$ satisfies
    \begin{equation*}
        \left(\mathcal{L}_{X_0}g\right)(X_i, X_j) = X_0\underbrace{g(X_i, X_j)}_{\delta_{ij}} - g(\underbrace{[X_0, X_i]}_{f_iX_0}, X_j) - g(X_i, \underbrace{[X_0, X_j]}_{f_jX_0}) = 0.
    \end{equation*}
    Thus $ \mathcal{L}_{X_0}(g\rvert_{Q \times Q}) \equiv 0$, meaning the flow of $X_0$ is isometric when restricted to the normal bundle $Q$, and therefore the flow of $X_0$ (the foliation $\mathcal{F} $) is Riemannian.
    \par Now let us show that each leaf $\gamma$ of $\mathcal{F} $ is a geodesic.
    To see this, note that $\gamma$ is a minimal submanifold of $M$. 
    The mean curvature of $\gamma$ in the normal direction $N \in \Gamma(Q)$ is the trace of the Weingarten map 
    \begin{align*}
        W^N: \ker\omega &\longrightarrow \ker\omega \\
        Z &\longmapsto W^N(Z) := \pi_{Q}\nabla_Z(N)
    \end{align*}
    where $\nabla$ is the Levi-Civita connection of $g$ and $\pi_{Q}: \mathrm{T}M \to Q$ the natural projection. 
    Since $Z = fX_0$ for some function $f: M \to \mathbb{R}$, we have
    \begin{equation*}
        W^N(Z) = f\pi_Q\nabla_{X_0}(N),
    \end{equation*}
    so it is sufficient for our needs to work with $X_0$.
    \par Fix a normal field $N \in \Gamma(Q)$.
    First, we observe that $\nabla_NX_0$ is orthogonal to $X_0$, since
    \begin{equation*}
        0 = \nabla_Ng(X_0,X_0) = 2g(\nabla_NX_0, X_0). 
    \end{equation*}
    In particular, 
    \begin{align*}
        \lambda_0([N, X_0]) &= g([N,X_0],X_0) = g(\nabla_NX_0, X_0) - g(\nabla_{X_0}N, X_0) \\ 
        &= -g(W^N(X_0), X_0) = -\mathrm{tr}W^N.   
    \end{align*}    
    Using the above relation, we obtain,
    \begin{align*}
        \iota_N\mathrm{d}\lambda_0(X_0) &= \mathrm{d}\lambda_0(N, X_0) \\ 
        &= N\underbrace{\lambda_0(X_0)}_{1} - X_0\underbrace{\lambda_0(N)}_{  0} - \lambda_0([N, X_0])  = \mathrm{tr} W^N
    \end{align*}
    that is, the interior product $\iota_N\mathrm{d}\lambda_0$ evaluated $X_0$ yields the mean curvature in the $N$ direction. 
    \par Finally, every normal vector is foliate, hence $[N, X_0] = fX_0$ for some function $f: M \to \mathbb{R}$, and
    \begin{equation*}
        \mathrm{tr} W^X = \iota_N\mathrm{d}\lambda_0(X_0) = \mathrm{d}\lambda_0(fX_0, X_0) = 0,
    \end{equation*}
    and the flow-line $\gamma$ is therefore, minimal.
    This implies, in turn, that $\gamma$ must be a geodesic.
    Indeed, $\gamma$ is already parameterized by arclength, since 
    \begin{equation*} 
        \dot\gamma(s) = \frac{d}{dt}\Bigr|_{t=0} \gamma(s+t) = X_0(\gamma(s))
    \end{equation*} 
    by definition, and $g(X_0, X_0) = 1$ by construction.
    Since $X_0$ has unit length, it follows that for any normal vector field $N$, one has 
    \begin{equation*}
        0 = \nabla_{\dot\gamma}g(N, \dot\gamma) = g(\nabla_{\dot\gamma}N, \dot\gamma) + g(N, \nabla_{\dot\gamma}\dot\gamma),
    \end{equation*}
    hence 
    \begin{equation*}
        \mathrm{tr}W^N = g(W^N(\dot\gamma), \dot\gamma) = - g(N, \nabla_{\dot\gamma}\dot\gamma),
    \end{equation*}
    and therefore $\gamma$ is a minimal submanifold if and only if $\nabla_{\dot\gamma}\dot\gamma$ is tangent to $\gamma$.
    But we already know that $\nabla_{\dot\gamma}\dot\gamma$ is normal to $\gamma$, since 
    \begin{equation*}
    0 = \nabla_{\dot\gamma}g(\dot\gamma,\dot\gamma) = 2g(\nabla_{\dot\gamma}\dot\gamma, \dot\gamma),
    \end{equation*}
    so the only way the vector field $\nabla_{\dot\gamma}\dot\gamma$ can be tangent to $\gamma$ is if it is zero, that is, if $\gamma$ is a geodesic of $g$.
    Therefore, the leaf $\gamma$ is minimal if and only if it is a geodesic.
\end{proof}
\begin{remark}
    As noted throughout the proof, every flow-line of $X_0$ is a minimal submanifold of $(M, g)$, which is the same as to say the foliation $\mathcal{F}$ is \emph{harmonic}, or \emph{geometrically taut}.
\end{remark}
\begin{remark}
    As pointed out in Remark \ref{lambda0notclosed}, in general $\mathrm{d}\lambda_0 \neq 0$.
    This implies that any of the fields $X_i, i = 0, \cdots\!, n-1$  preserves any of the forms $\lambda_j, j = 1, \cdots\!, n-1$, since, by Cartan's formula,
    \begin{equation*}
        \mathcal{L}_{X_i}\lambda_j = \mathrm{d}(\underbrace{\iota_{X_i}\lambda_j}_{\delta_{ij}}) + \iota_{X_i}\underbrace{\mathrm{d}\lambda_0}_{0} = 0,
    \end{equation*}
    while on the other hand, for $\lambda_0$ this Lie derivative results in $\iota_{X_i}\mathrm{d}\lambda_0$. 
    As a consequence, we see that the metric $g = \sum_j \lambda_j\otimes\lambda_j$ satisfies
    \begin{equation*}
        \mathcal{L}_{X_0}g = \sum_{j = 0}^{n-1}\left(\mathcal{L}_{X_0}\lambda_j \otimes \lambda_j + \lambda_j \otimes \mathcal{L}_{X_0}\lambda_j\right) = \mathcal{L}_{X_0}\lambda_0 \otimes \lambda_0 + \lambda_0 \otimes \mathcal{L}_{X_0}\lambda_0.
    \end{equation*}
    So, in general, the flow of $X_0$ is Riemannian and geodesible with respect to $g$, but not isometric.
    We can, however, using Theorem \ref{int_har_char2}, find a metric tensor $g_0$ with respect to which $X_0$ is Killing, and therefore the flow of $X_0$ is equicontinuous with respect to $g_0$.
    Since any two metrics one $M$ are quasi-isometric (due to compactness), it follows that the flow of $X_0$ is also equicontinuous with respect to $g$.
    \end{remark}
\par We are finally able to prove Theorem \ref{main_thm}.
\begin{theorem}\label{torus_char}
    If $[\omega]\neq 0$ in $H^{n-1}_{dR}(M)$ then $M$ and  $\mathbb{T}^n$ are diffeomorphic. 
\end{theorem}
\begin{proof}
    As noted in the Introduction, it is sufficient to ensure the existence of a closed form $\eta$ such that the set $\{\omega_1,\cdots\!,\omega_{n-1},\eta\}$ is linearly independent.
    Since, according to Theorem \ref{riem_geod_flow}, the flow induced by $\omega$ (that is, the flow of $X_0$) is Riemannian and geodesible, and $[\omega]\neq 0$ in $H^{n-1}_{dR}(M)$ by hypothesis, it follows from Theorem \ref{int_har_char2} that $\omega$ is intrinsically harmonic, say harmonic for a metric $g_0$.
    As the Riemannian metric $g_0$ is such that $ *_{g_0}\omega$ is closed, the required form is given by $\eta=*_{g_0}\omega$.
\end{proof}
\begin{remark}
    One could also obtain Theorem \ref{main_thm} as follows.
    By Tischler's argument \cite{tischler_fibering_1970}, there exists a sequence of fibrations 
    \begin{equation*}
        \xi_n=(\pi_n,M,\mathbb{T}^{n-1})
    \end{equation*}
    such that $\pi_n^*(\Omega)\to \omega$, where $\Omega$ is a fixed volume form on $\mathbb{T}^{n-1}$, and the convergence is in the sense of De Rham currents.
    Since $[\omega]\neq 0$, the same is true when $\omega$ is viewed as current.
    On the other hand, the boundary subspace is closed in the current space and $\omega$ is null in cohomology if, and only if, is a boundary (as current).
    It follows that eventually $\pi_n^*(\Omega)$ determines a nonzero cohomological class.
    Let $m\in\mathbb{N}$ such that $\pi_m^*(\Omega)$ determines a nonzero cohomological class.
    It follows that the Euler characteristic class  $\chi(\xi_m)$ of $\xi_m$ is a torsion element (see \cite{franca_intrinsically_nodate,miyoshi2001remark}).
    On the other hand, it is well known that the circle bundles over $\mathbb{T}^{n-1}$ is in one-to-one correspondence with $H^2(\mathbb{T}^{n-1},\mathbb{Z})$, and the null class correspond to the trivial circle bundle.
    Since $H^2(\mathbb{T}^{n-1},\mathbb{Z})$ is torsion free, it follows that $\xi_m$ is trivial, hence $M$ is diffeomorphic to $\mathbb{S}^1\times\mathbb{T}^{n-1} \approx \mathbb{T}^n$, as desired. 
\end{remark}
\begin{remark}
    Let $M$ be a manifold with a closed-decomposable form 
    \begin{equation*}
        \omega=\omega_1\wedge\cdots\wedge\omega_p
    \end{equation*} 
    (i.e., each $\omega_i$ is a closed $1$-form).
    Again, by Tischler's argument, there exists a sequence of integers $(d_i)$ and a collection of smooth fiber bundles 
    \begin{equation*}
        \mathcal{B}_l=\{M,p_l,\mathbb{T}^p,F_l\}
    \end{equation*}
    such that 
    \begin{equation*}
        \omega_l=\frac{1}{d_l}p_l^*\Omega_{\mathbb{T}^p}\to \omega \text{ in } \mathcal{D}_p.
    \end{equation*}
    In particular, if $\omega$ determines a nonzero cohomological class, then the forms $\pi_l^*\Omega_{\mathbb{T}^n}$ are eventually not cohomologous to zero.
    Suppose $\omega$ is harmonic with respect to a Riemannian metric $g$ on $M$.
    Then $\eta=*_g\omega$ is a closed form of rank $(n-p)$.
    Since  $\omega_l\to \omega$ in $\mathcal{D}_p$ and $\omega\wedge\eta>0 $,  eventually we have  $\omega_l\wedge\eta>0$. 
    By \cite[Theorem 1]{franca_intrinsically_nodate} it follows that $\omega_l$ is intrinsically harmonic.
    From this point of view, the problem of whether closed-decomposable $p$-forms are intrinsically harmonic is equivalent to the study of closed-decomposable $p$-forms $\pi^*\Omega_{\mathbb{T}^p}$ in fiber bundles $\mathcal{B}=\{M,p,\mathbb{T}^p,F\}$ over the $p$-torus.
    This paper provides a positive answer for such a problem in the case where $p=n-1$. 
    In general, an answer to the problem is not known.
    It is also noteworthy that the existence of an linearly independent set consisting of $(n-1)$ $1$-forms do not imply that we have a $n$-torus. 
    For instance, every three-dimensional manifold is parallelizable, hence it admits a globally defined set consisting of two linearly independent $1$-forms (see \cite{milnor1974characteristic}).
\end{remark}

\section{Conclusion}
\par Initially, we conjectured that every geodesible flow on a compact manifold had to be equicontinuous. 
As it turned out, this is not the case, as counterexamples from Remark \ref{counterexamples} show. 
We were able, however, to provide a weaker characterization in the form of Theorem \ref{int_har_char2}, which allowed us, together with the theory exposed in the previous sections, to obtain the characterization of the $n$-torus in Theorem \ref{main_thm}. 
\par Our current goal is to characterize volume-preserving flows admitting a global cross-section. 
An immediate necessary condition for a volume-preserving flow to have one such cross-section is that a natural class determined by the flow must be nontrivial. 
In this paper, we show that, for a very particular case, the converse is true, and provide strategies that can be useful to approach the general problem. 
An affirmative answer to the general case would also lead us to a dual version of Tischler's theorem,  which we conjecture to be true. 

%%%%References

%\bibliographystyle{plain}
%\bibliography{references}

\end{document}